\documentclass[10pt, reqno]{amsart}
\usepackage{graphicx, amssymb, amsmath, amsthm}
\numberwithin{equation}{section}

\usepackage{color}

\usepackage[backref=page]{hyperref}  
\usepackage{amscd}
\def\pa{\partial}

\let\Re=\undefined\DeclareMathOperator*{\Re}{Re}
\let\Im=\undefined\DeclareMathOperator*{\Im}{Im}

\newcommand{\R}{\mathbb{R}}
\newcommand{\C}{\mathbb{C}}

\newcommand{\la}{\mathcal{L}_a}

\newtheorem{theorem}{Theorem}[section]

\newtheorem{lemma}[theorem]{Lemma}

\newtheorem{proposition}[theorem]{Proposition}

\theoremstyle{definition}

\newtheorem{remark}[theorem]{Remark}

\makeatletter
\newcommand{\Extend}[5]{\ext@arrow0099{\arrowfill@#1#2#3}{#4}{#5}}
\makeatother

\begin{document}
\title[3d cubic INLS]{Scattering theory for 3d cubic inhomogeneous NLS with inverse square potential}

\author{Ying Wang}
\address{The Graduate School of China Academy of Engineering Physics,  \ Beijing, \ China, \ 100088}
\email{wsming@bupt.cn}

\begin{abstract}
In this paper, we study the scattering theory for the cubic inhomogeneous Schr\"odinger equations with inverse square potential $iu_t+\Delta u-\frac{a}{|x|^2}u=\lambda |x|^{-b}|u|^2u$ with $a>-\frac14$ and $0<b<1$ in dimension three. In the defocusing case (i.e. $\lambda=1$), we establish the global well-posedness and scattering for any initial data in the energy space $H^1_a(\R^3)$. While for the focusing case(i.e. $\lambda=-1$), we obtain the scattering for the initial data below the threshold of the ground state, by making use of the virial/Morawetz argument as in Dodson-Murphy \cite{DM1} and Campos-Cardoso \cite{CC} that avoids the use of interaction Morawetz estimate.

\end{abstract}

 \maketitle

\begin{center}
 \begin{minipage}{100mm}
 \end{minipage}
 \end{center}


\section{Introduction}

\noindent
In this paper, we study the Cauchy problem for the cubic inhomogeneous Schr\"odinger equations with inverse square potential  of the form
\begin{equation}\label{equ:nls}
  \begin{cases}
  i\pa_tu-\la u=\lambda |x|^{-b}|u|^2u,\quad (t,x)\in\R\times\R^3\\
  u(0,x)=u_0(x)\in H^1(\R^3),
  \end{cases}
\end{equation}
where $u:\;\R\times\R^3\to\C,\;0<b<1,\;\la=-\Delta+\tfrac{a}{|x|^2}$ with $a>-\tfrac14,\;\lambda\in\{\pm1\},$   $\lambda=1$ corresponding to the defocusing case, and $\lambda=-1$ corresponding to the focusing case.
The scale-covariance elliptic operator $\la$
appearing in \eqref{equ:nls} plays a key role  in many problems of
physics and geometry. Especially, the heat and Schr\"odinger flows for the
elliptic operator $\la$ have been studied in the
theory of combustion (see \cite{LZ}), and in quantum mechanics (see
\cite{KSWW}).

Equation \eqref{equ:nls} admits a number of symmetries in energy space $H^1$, explicitly:

$\bullet$ {\bf Time translation invariance:} if $u(t,x)$ solves \eqref{equ:nls}, then so does $u(t+t_0,x),~t_0\in\R$;

$\bullet$ {\bf Phase invariance:} if $u(t,x)$ solves \eqref{equ:nls}, then so does $e^{i\gamma}u(t,x),~\gamma\in\R.$

From the Ehrenfest law or direct computation, these two symmetries induce invariances in
the energy space $H^1(\R^3)$, namely
\begin{align*}
&{\rm Mass:}~~M(u)=\int_{\R^3}|u(t,x)|^2\;dx=M(u_0);\\
&{\rm Energy:}~~E_a(u)=\int_{\R^3}\Big(\frac12|\nabla u(t,x)|^2+\frac{a}{2}\frac{|u|^2}{|x|^2}+\frac{\lambda}{4}\frac{|u(t,x)|^{4}}{|x|^b}\Big)\;dx=E_a(u_0).
\end{align*}

$\bullet$ {\bf Scaling invariance:} if $u(t,x)$ solves \eqref{equ:nls}, then so does $u_\mu(t,x)$ defined by
\begin{equation}\label{equ:scale}
u_\mu(t,x)=\mu^{\frac{2-b}2}u(\mu^2t, \mu x),\quad\mu>0.
\end{equation}
This scaling defines a notion of \emph{criticality} for \eqref{equ:nls}. In particular, one can check that
the only homogeneous $L_x^2$-based Sobolev space that is left invariant under \eqref{equ:scale} is $\dot{H}_x^{s_c}(\R^3)$, where the \emph{critical regularity} $s_c$ is given by $s_c:=\frac{1+b}2$,
$$\big\|u_\mu(t,\cdot)\big\|_{\dot
H^{s_c}_x(\R^3)}=\|u(\mu^2t,\cdot)\|_{\dot H^{s_c}_x(\R^3)}.$$
For $s_c=1$, we call the problem \eqref{equ:nls} \emph{energy critical}. For $s_c=0$, we call the problem \emph{mass critical}, while for $0<s_c<1$ we call the problem \emph{mass supercritical} and \emph{energy subcritical}. In this paper, we will focus on the case $\tfrac12<s_c<1$, i.e. $0<b<1$. To state our main result, we first recall the history of the study for the problem  \eqref{equ:nls}.

When $a=0$, we write the problem \eqref{equ:nls} as
\begin{equation}\label{equ:inls}
  i\pa_tu+\Delta u=\lambda|x|^{-b}|u|^2u,\quad (t,x)\in\R\times\R^3.
\end{equation} This model can be thought of as modeling inhomogeneities in
the medium in which the wave propagates, see \cite{KMVBT}. Especially, when $b=0$, the local and global well-posedness for the Schr\"odinger equation \eqref{equ:inls} has been intensively studied in
\cite{C21,Cav,CaW,CKSTT04,DM1,DM2,DHR,HR,TVZ}. First, Cazenave and  Weissler \cite{CaW} proved that
\eqref{equ:inls} with $b=0$ is locally wellposed in $H^1(\R^3)$. On the other
hand, since the lifespan of the local solution depends only on the
$H^1$-norm of the initial data, one can easily obtain the global
wellposedness for \eqref{equ:inls} with $b=0$
from the conservation of mass and energy when $\lambda=1$. Furthermore,  Ginibre and Velo \cite{GV85} proved
the scattering  by making use of the
almost finite propagation speed
$$\int_{|x|\geq a} |u(t,x)|^2 \; dx \leq \int \min\left(\frac{|x|}{a},
1\right) |u(t_0)|^2 \; dx + \frac{C}{a} \cdot |t-t_0|$$ for large
spatial scale and the classical  Morawetz inequality in  \cite{LS}
\begin{equation}\label{cmrsz}
\iint_{\R\times \R^3}\frac{|u(t,x)|^{4}}{|x|} \; dtdx \lesssim
\|u\|_{L_t^\infty\dot H^\frac12}^2\lesssim C\big(M(u_0), E_0(u_0)\big)
\end{equation} for small spatial scale. Lately, one can give another simple proof by using I-term's interaction Morawetz estimate\cite{CKSTT04}
\begin{equation}\label{equ:intmor}
  \|u\|_{L_{t,x}^4(\R\times\R^3)}\leq C\big(M(u_0), E_0(u_0)\big).
\end{equation}
 While for the focusing case $\lambda=-1$, under the assumption
\begin{equation}\label{equ:assum3dnls}
  M(u_0)E_0(u_0)<M(Q)E_0(Q)\quad\text{and}\quad \|u_0\|_{L^2}\|u_0\|_{\dot{H}^1}<\|Q\|_{L^2}\|Q\|_{\dot{H}^1}
\end{equation}
with $Q$ being the ground state of the elliptic equation $-\Delta Q+Q=Q^3$, Holmer and Roudenko \cite{HR} utilized concentration compactness/rigity method to obtain the global well-posedness and scattering result for the radial initial date. Lately, this was extended to the nonradial inital data by Duyckaerts-Holmer-Roudenko \cite{DHR}. Recently, Dodson-Murphy \cite{DM1,DM2} gave another simple proof that avoids the use of concentration compactness by making use of virial/Morawetz estimate. When $b\neq0$, Farah and Guzm\'an  \cite{FG} established the  scattering result for the radial initial date below the ground state. This was extended to the nonradial initial data by
 Miao, Murphy, and Zheng \cite{MMZ} by exploiting the embedding of nonlinear profile. Recently, Campos and Cardoso \cite{CC} gave a simple proof of the result in \cite{MMZ} by using Dodson-Murphy's new method in \cite{DM1} and developing the decay $|x|^{-b}$ in the nonlinear term that avoids the use of interaction Morawetz estimate. In this paper, we will utilize the argument as in \cite{CC} to consider the case that $a>-\tfrac14$ and $0<b<1$.

 When $b=0$, we can write the equation \eqref{equ:nls} as
 \begin{equation}\label{equ:nlsa}
   i\pa_tu-\la u=\lambda|u|^2u,\quad (t,x)\in\R\times\R^3.
 \end{equation}
This model appears in several physical settings, such a quantum field equations or black hole solutions of the
Einstein's equations (see \cite{KSWW} and the
references therein). Fanelli, Felli, Fontelos, and Primo proved in \cite{FFFP1} investigated the validity of the time-decay estimate for the Schr\"odinger evolution and proved that it holds, in some specific cases, including the inverse square potential. The Strichartz estimates, an essential tool for studying the behaviour of solutions to nonlinear dispersive equations, have been developed by
Burq-Planchon-Stalker-Tahvildar-Zadeh \cite{BPSS} by using the perturbation method. Zhang-Zheng
\cite{ZZ} studied the defocusing case($\lambda=1$), establishing well posedness and scattering for $a\geq0.$ While for the focusing case: $\lambda=-1$, Killip, Visan, Murphy and Zheng \cite{KMVZ} generalized Duyckaerts-Holmer-Roudenko's result of \cite{DHR} to the problem \eqref{equ:nlsa} with $\lambda=-1$ by assuming that
\begin{equation}\label{equ:assumnlsa}
  M(u_0)E_a(u_0)<M(Q_{a\wedge 0})E_{a\wedge 0}(Q_{a\wedge 0}),\;\|u_0\|_{L_x^2}\| u_0\|_{\dot H_a^1}<\|Q_{a\wedge 0}\|_{L_x^2}\|Q_{a\wedge 0}\|_{\dot H_{a\wedge 0}^1},
\end{equation}
where $A\wedge B = \min\{A,B\}$, and  $Q_a$ solves the elliptic problem
\begin{equation}\label{ell}
\la Q_a + Q_a =|Q_a|^2Q_a,
\end{equation}
and
\begin{equation}\label{equ:doth1def}
  \|f\|_{\dot{H}^1_a}^2:=\big\|\sqrt{\la}f\big\|_{L_x^2(\R^3)}^2=\int_{\R^3}\big(|\nabla f|^2+a\tfrac{|f|^2}{|x|^2}\big)\;dx.
\end{equation}
We remark that the assumption \eqref{equ:assumnlsa} follows from the sharp Gagliardo-Nirenberg inequality:
\begin{equation}\label{equ:gna}
\|f\|_{L_x^{4}(\R^3)}^4 \leq C_a
\|f\|_{L_x^2(\R^3)}
\|\sqrt{\la}f\|_{L_x^2(\R^3)}^3.
\end{equation}

When $a\neq0$ and $b\neq0$, the problem \eqref{equ:nls} is a model from various physical contexts, for
example, in nonlinear optical systems with spatially dependent interactions (see \cite{BPVT} and the
references therein).  Recently, using the following Gagliardo-Nirenberg-type estimate
\begin{equation}\label{equ:gnab}
  \int_{\R^3}\frac{|f(x)|^4}{|x|^b}\;dx\leq C_a\|f\|_{L_x^2(\R^3)}^{1-b}
\|\sqrt{\la}f\|_{L_x^2(\R^3)}^{3-b},\;0<b<1,
\end{equation}  Campos and Guzm\'an \cite{CG} established sufficient conditions for global
existence and blow-up in the energy space $H^1_a=\dot{H}^1_a\cap L^2$ for the focusing case: $\lambda=-1$, by assuming that
\begin{equation}\label{equ:u0iniassum}
  M(u_0)E_a(u_0)<M(Q_{a,b})E_a(Q_{a,b}).
\end{equation}
Furthermore, they also showed that the equality in \eqref{equ:gnab} above is attained by a function $Q_{a,b}\in H^1_a$, which is a positive solution to the
elliptic equation
\begin{equation}\label{equ:qaelliptic}
  \la Q_{a,b}+Q_{a,b}=|x|^{-b}|Q_{a,b}|^2Q_{a,b}.
\end{equation}
In this paper, we aim to study the scattering theory for the problem \eqref{equ:nls} in the defocusing case $\lambda=1$, and establish the scattering result for the global solution obtained in \cite{CG} for the focusing case $\lambda=-1$. More precisely, our main result is as follows.

\begin{theorem}[Main result]\label{thm:main}
Let $a>-\tfrac14,\;0<b<1,\;\lambda\in\{\pm1\},\;u_0\in H^1_a(\R^3)$. In the focusing case $\lambda=-1$, we also assume that
\begin{equation}\label{equ:assumpinit}
  M(u_0)E_a(u_0)<M(Q_{a,b})E_a(Q_{a,b}),\;\|u_0\|_{L^2}\|u_0\|_{\dot{H}^1_a}<\|Q_{a,b}\|_{L^2}\|Q_{a,b}\|_{\dot{H}^1_a}.
\end{equation}
Then, there is a unique global solution $u\in C(\R,H^1_a)\cap L_t^4L_x^6(\R\times\R^3)$ to \eqref{equ:nls} and it scatters in the sense that there exists $u_\pm\in H^1_a(\R^3)$ such that
\begin{equation}\label{equ:scattesinse}
  \lim_{t\to\pm\infty}\big\|u(t,\cdot)-e^{it\la}u_\pm\big\|_{H^1_a}=0.
\end{equation}

\end{theorem}

\begin{remark}
$(i)$ By the same argument as in this paper, one can also obtain the similar result for the general nonlinear term $|x|^{-b}|u|^{p-1}u$. The restriction for dimension three stems from the dispersive estimate for the Schr\"odinger operator $e^{it\la}$, which was established in \cite{FFFP1} for dimension three. The dispersive estimates in higher dimensions large than four are still open. And the dispersive estimate for the radial initial data was obtained in Zheng \cite{Zheng}. Combining such dispersive estimate with the argument as in this paper, one can get the similar result in higher dimensions but for the radial initial data.

$(ii)$ In Appendix, we will first establish the interaction Morawetz estimate in the defocusing case. As an application, we give a simple proof for the scattering theory for \eqref{equ:nls} with $\lambda=1$ but $a>0$. While in Theorem \ref{thm:main}, we can get the scattering result for the negative $a$ by making use of the virial/Morawetz argument.
\end{remark}

\underline{\bf Outline of the proof:}  By using Strichartz estimates, the fixed point argument, mass/energy conservation and variational analysis, Campos and Guzm\'an \cite{CG} have established the global solution for the problem \eqref{equ:nls} under the assumption \eqref{equ:assumpinit}. Therefore, we only need to show the scattering part of Theorem \ref{thm:main}. To do this, we will establish a scattering criterion by following the argument as in \cite{CC,DM1,Tao}.
\begin{lemma}[Scattering criterion]\label{lem:scatcriterion}
Suppose $u:~\R\times\R^3\to\mathbb{C}$ is a global solution to
\eqref{equ:nls} satisfying
\begin{equation}\label{equ:uhasuni}
\sup_{t\in\R}\|u\|_{H^1_a(\R^3)}\leq E.
\end{equation}
There exist $\epsilon=\epsilon(E)>0$ and $R=R(E)>0$ such that if
\begin{equation}\label{equ:masslim}
\liminf\limits_{t\to\infty}\int_{|x|<R}|u(t,x)|^2\;dx\leq\epsilon^2,
\end{equation}
then, $u$ scatters forward in time.
\end{lemma}

By Lemma \ref{lem:scatcriterion}, we are reduced to verify the assumption \eqref{equ:masslim}.  By H\"older's inequality, \eqref{equ:masslim}  follows from the following lemma.
\begin{lemma}\label{lem:poentialenergdec}
There exists a sequence of times $t_n\to\infty$ and a sequence of radii $R_n\to\infty$ such that
\begin{equation}\label{equ:potendecay}
\lim_{n\to\infty}\int_{|x|\leq R_n}\frac{|u(t_n,x)|^4}{|x|^b}\;dx=0.
\end{equation}
\end{lemma}
By the Mean Value Theorem,
 the above lemma can be derived by the following proposition(choosing $T$ sufficiently large and $R=T^\frac{1}{b+1}$).
\begin{proposition}[Morawetz estimate]\label{prop:virmorest}
Let $T>0.$ For $R=R(\delta,M(u_0),Q_{a,b})$ sufficiently large, we have
\begin{equation}\label{equ:potensmal}
\frac1{T}\int_0^T\int_{|x|\leq \frac{R}4}|u(t,x)|^{p+1}\;dx\;dt\lesssim\frac{R}{T}+\frac1{R^b}.
\end{equation}
\end{proposition}
We will use the virial/Morawetz argument as in \cite{CC,DM1} to prove this proposition in Section \ref{sec:vme}. Combining Lemma \ref{lem:scatcriterion} and Proposition \ref{prop:virmorest}, we conclude the proof of Theorem \ref{thm:main}.

The paper is organized as follows. In Section \ref{sec:pre},  as
preliminaries, we collect some useful lemmas, including some harmonic analysis tools from \cite{KMVZZ} related to the operator $\la$, Strichartz estimates and some nonlinear estimates, and the variational analysis for the sharp Gagliardo--Nirenberg inequality \eqref{equ:gnab}.   In Section \ref{sec:scatcrit}, we will
 prove Lemma  \ref{lem:scatcriterion} by following the argument as in \cite{CC,DM1,Tao}.
  We prove Proposition \ref{prop:virmorest} in Section \ref{sec:vme} based
on the virial/Morawetz argument. Finally, in Appendix, we will derive the interaction Morawetz estimate in the defocusing case.




\section{Preliminaries}\label{sec:pre}

We begin by introducing some notation.  To simplify the expression of our inequalities, we introduce
some symbols $\lesssim, \thicksim, \ll$. If $X, Y$ are nonnegative
quantities, we use $X\lesssim Y $ or $X=O(Y)$ to denote the estimate
$X\leq CY$ for some $C$, and $X \thicksim Y$ to denote the estimate
$X\lesssim Y\lesssim X$. We use $X\ll Y$ to mean $X \leq c Y$ for
some small constant $c$. We use $C\gg1$ to denote various large
finite constants, and $0< c \ll 1$ to denote various small
constants. For any $r, 1\leq r \leq \infty$, we denote by $\|\cdot
\|_{r}$ the norm in $L^{r}=L^{r}(\mathbb{R}^3)$ and by $r'$ the
conjugate exponent defined by $\frac{1}{r} + \frac{1}{r'}=1$. We
denote $a_{\pm}$ to be any quantity of the form $a\pm\epsilon$ for
any $\epsilon>0$.

\subsection{Harmonic analysis for $\la$}\label{subsec:harmonic}

For $1< r < \infty$, we write $\dot H^{1,r}_a(\R^3)$ and $
H^{1,r}_a(\R^3)$ for the homogeneous and inhomogeneous Sobolev
spaces associated with $\la$, respectively, which have norms
$$
\|f\|_{\dot H^{1,r}_a(\R^3)}= \|\sqrt{\la} f\|_{L^r(\R^3)} \quad\text{and}\;
\|f\|_{H^{1,r}_a(\R^3)}= \|\sqrt{1+ \la} f\|_{L^r(\R^3)}.
$$
When $r=2$, we simply write $\dot H^{1}_a(\R^3)=\dot
H^{1,2}_a(\R^3)$ and $H^{1}_a(\R^3)=H^{1,2}_a(\R^3)$.

Now, we recall a key proposition which states that Sobolev spaces defined through powers of $\la$ coincide with the traditional Sobolev spaces.

\begin{proposition}[Equivalence of Sobolev spaces, \cite{KMVZZ}]\label{pro:equivsobolev} Let $a\geq -\tfrac14,\;\sigma=\tfrac12-\sqrt{\tfrac14+a}$, and $0<s<2$. If $1<p<\infty$ satisfies $\frac{s+\sigma}{3}<\frac{1}{p}< \min\{1,\frac{3-\sigma}{3}\}$, then
\[
\||\nabla|^s f \|_{L_x^p}\lesssim_{p,s} \|(\la)^{\frac{s}{2}} f\|_{L_x^p}\quad\text{for all}\;f\in C_c^\infty(\R^3\backslash\{0\}).
\]
If $\max\{\frac{s}{3},\frac{\sigma}{3}\}<\frac{1}{p}<\min\{1,\frac{3-\sigma}{3}\}$, then
\[
\|(\la)^{\frac{s}{2}} f\|_{L_x^p}\lesssim_{p,s} \||\nabla|^s f\|_{L_x^p} \quad\text{for all}\; f\in C_c^\infty(\R^3\backslash\{0\}).
\]
\end{proposition}

We will also  use  the following fractional calculus estimates due to Christ and Weinstein \cite{CW}.  Combining these estimates with Lemma~\ref{pro:equivsobolev}, we can deduce analogous statements for the operator $\la$ (for restricted sets of exponents).
\begin{lemma}[Fractional calculus]\label{lem:frac}
\begin{itemize}
\item[(i)] Let $s\geq 0$ and $1<r,r_j,q_j<\infty$ satisfy $\tfrac{1}{r}=\tfrac{1}{r_j}+\tfrac{1}{q_j}$ for $j=1,2$. Then
\[
\big\||\nabla|^s(fg) \big\|_{L_x^r} \lesssim \|f\|_{L_x^{r_1}} \||\nabla|^s g\|_{L_x^{q_1}} + \| |\nabla|^s f\|_{L_x^{r_2}} \| g\|_{L_x^{q_2}}.
\]
\item[(ii)] Let $G\in C^1(\C)$ and $s\in (0,1]$, and let $1<r_1\leq \infty$  and $1<r,r_2<\infty$ satisfy $\tfrac{1}{r}=\tfrac{1}{r_1}+\tfrac{1}{r_2}$. Then
\[
\big\||\nabla|^s G(u)\big\|_{L_x^r} \lesssim \|G'(u)\|_{L_x^{r_1}} \|u\|_{L_x^{r_2}}.
\]
\end{itemize}
\end{lemma}

 \subsection{Dispersive and Strichartz estimates}\label{sze}
In \cite{FFFP1}, Fanelli,  Felli, Fontelos, and  Primo showed the following dispersive estimate for the Schr\"odinger operator $e^{it\la}$.
\begin{lemma}[Dispersive estimate, \cite{FFFP1}]\label{thm:disp}
$(i)$ If $a\geq0$, then
\begin{equation}\label{equ:dispapos}
\|e^{it\la}f\|_{L^\infty(\R^3)}\leq C|t|^{-\frac{3}2}\|f\|_{L_x^1(\R^3)}.
\end{equation}

$(ii)$ If $-\tfrac14<a<0$, then there holds
\begin{equation}\label{equ:dispaneg}
\big\|(1+|x|^{-\sigma})^{-1}e^{it\la}f\big\|_{L^\infty(\R^3)}\leq C\frac{1+|t|^{\sigma}}{|t|^{\frac{3}2}}\big\|(1+|x|^{-\sigma})f\big\|_{L_x^1(\R^3)},
\end{equation}
with $\sigma=\tfrac12-\sqrt{\tfrac14+a}$ being as in Proposition \ref{pro:equivsobolev}.
\end{lemma}

Strichartz estimates for the propagator $e^{it\la}$  were proved by Burq, Planchon, Stalker, and Tahvildar-Zadeh in \cite{BPSS}. Zhang-Zheng \cite{ZZ17} confirmed the
double endpoint case. Combining these with the Christ--Kiselev Lemma \cite{CK}, we obtain the following Strichartz estimates:

\begin{proposition}[Strichartz, \cite{BPSS,ZZ17}] Fix $a>-\frac14$. The solution $u$ to $(i\partial_t+\la)u = F$ on an interval $I\ni t_0$ obeys
\[
\|u\|_{L_t^q L_x^r(I\times\R^3)} \lesssim \|u(t_0)\|_{L_x^2(\R^3)} + \|F\|_{L_t^{\tilde q'} L_x^{\tilde r'}(I\times\R^3)}
\]
for any $2\leq q,\tilde q\leq\infty$ with $\frac{2}{q}+\frac{3}{r}=\frac{2}{\tilde q}+\frac{3}{\tilde r}= \frac32$.
\end{proposition}

\subsection{Variational analysis}\label{subsec:variana}
In this subsection,
 we recall the variational analysis for the sharp Gagliardo--Nirenberg inequality \eqref{equ:gnab} in \cite{CG}.

\begin{proposition}[Coercivity, \cite{CG}]\label{P:coercive} Fix $a>-\tfrac14$. Let $u:I\times\R^3\to\C$
be the maximal-lifespan solution to \eqref{equ:nls} with $u(t_0)=u_0\in H_a^1\backslash\{0\}$ for some $t_0\in I$. Assume that
\begin{equation}\label{equ:assumpu01}
M(u_0)E_a(u_0) \leq (1-\delta)M(Q_{a,b})E_a(Q_{a,b})\quad\text{for some}\quad\delta>0.
\end{equation}
Then there exist $\delta'=\delta'(\delta)>0$, $c=c(\delta,a,\|u_0\|_{L_x^2})>0$, and $\epsilon=\epsilon(\delta)>0$ such that:
If $\|u_0\|_{L_x^2}\| u_0\|_{\dot H_a^1} \leq \|Q_{a,b}\|_{L_x^2} \|Q_{a,b}\|_{\dot H_a^1}$, then for all $t\in I$,
\begin{itemize}
\item[(i)] $\|u(t)\|_{L_x^2} \|u(t)\|_{\dot H_a^1} \leq (1-\delta')\|Q_{a,b}\|_{L_x^2} \|Q_{a,b}\|_{\dot H_a^1}$,
\item[(ii)] $\|u(t)\|_{\dot H_a^1}^2 -  \int_{\R^3}\frac{|u|^4}{|x|^b}\;dx \geq c\int_{\R^3}\frac{|u|^4}{|x|^b}\;dx.$
\end{itemize}

\end{proposition}




\section{Proof of Lemma \ref{lem:scatcriterion}}\label{sec:scatcrit}
In this section, we will prove Lemma \ref{lem:scatcriterion}, that is
\begin{lemma}[Scattering criterion]\label{lem:scatcriterion-1}
Suppose $u:~\R\times\R^3\to\mathbb{C}$ is a global solution to
\eqref{equ:nls} satisfying
\begin{equation}\label{equ:uhasuni-1}
\sup_{t\in\R}\|u\|_{H^1_a(\R^3)}\leq E.
\end{equation}
There exist $\epsilon=\epsilon(E)>0$ and $R=R(E)>0$ such that if
\begin{equation}\label{equ:masslim-1}
\liminf\limits_{t\to\infty}\int_{|x|<R}|u(t,x)|^2\;dx\leq\epsilon^2,
\end{equation}
then, $u$ scatters forward in time.
\end{lemma}

\begin{proof}
It is equivalent to show that
\begin{equation}\label{equ:scatsize}
  \|u\|_{L_t^4L_x^6([0,\infty)\times\R^3)}<+\infty.
\end{equation}
By H\"older's inequality in time, Sobolev embedding and \eqref{equ:uhasuni-1}, we obtain for any $T<+\infty$
\begin{equation*}
   \|u\|_{L_t^4L_x^6([0,T)\times\R^3)}\leq T^\frac14 \|u\|_{L_t^\infty \dot{H}^1([0,\infty)\times\R^3)}\lesssim T^\frac14.
\end{equation*}
Hence, it suffices to prove that for large time $T$
\begin{equation}\label{equ:scatsizelarT}
  \|u\|_{L_t^4L_x^6([T,\infty)\times\R^3)}<+\infty.
\end{equation}
Using the local theory, this will follow from
\begin{equation}\label{equ:linimpnonl}
  \big\|e^{i(t-T)\la}u(T)\big\|_{L_t^4L_x^6([T,\infty)\times\R^3)}\ll1.
\end{equation}
By Duhamel's formula, we have
\begin{align}\nonumber
e^{i(t-T)\la}u(T)=&e^{it\la}u_0+i\int_0^T e^{i(t-s)\la}\big(|x|^{-b}|u|^2u\big)(s)\;ds\\\label{equ:duhforeit}
\triangleq&e^{it\la}u_0+F_1(t,x)+F_2(t,x),
\end{align}
where
$$F_j(t,x):=i\int_{I_j} e^{i(t-s)\la}\big(|x|^{-b}|u|^2u\big)(s)\;ds,$$
and
$$I_1=[0,T-\epsilon^{-\alpha}],\;I_2=[T-\epsilon^{-\alpha},T],$$
for some $\alpha>0$ to be determined lately.

From Strichartz estimates and \eqref{equ:uhasuni-1},  we know that
$$\|e^{it\la}u_0\|_{L_t^4L_x^6([0,\infty)\times\R^3)}\lesssim\|u_0\|_{\dot{H}^\frac12_a}\lesssim E.$$
This together with absolutely continuous yields for sufficiently large $T$
\begin{equation}\label{equ:linpsma}
  \|e^{it\la}u_0\|_{L_t^4L_x^6([T,\infty)\times\R^3)}\leq \epsilon.
\end{equation}

{\bf Estimate of the term $F_1(t,x)$:} Using Duhamel's formula, we can rewrite $F_1(t)$ as
\begin{equation}\label{equ:f2rw}
F_1(t)=e^{i(t-T+\epsilon^{-\alpha})\la}\big[u(T-\epsilon^{-\alpha})\big]-e^{it\la}u_0.
\end{equation}
By Strichartz estimates and \eqref{equ:uhasuni-1}, we get
\begin{equation}\label{equ:l4l6est}
\|F_1\|_{L_t^2 L_x^6\cap L_t^\infty L_x^2([T,\infty)\times\R^3)}\lesssim\|u_0\|_{L_x^2}\lesssim E.
\end{equation}
On the other hand, by Lemma \ref{thm:disp}, we get
\begin{align}\nonumber
\|F_1(t)\|_{L_x^6(|x|\leq
1)}\lesssim&\int_{I_1}\big\|e^{i(t-s)\la}\big(|x|^{-b}|u|^2u\big)(s)\big\|_{L_x^6(|x|\leq
1)}\;ds\\\nonumber
\lesssim&\int_{I_1}\big\|(1+|x|^{-\sigma_1})^{-1}e^{i(t-s)\la}\big(|x|^{-b}|u|^2u\big)\big\|_{L_x^\infty}\;ds
\cdot\|(1+|x|^{-\sigma_1})\|_{L_x^6(|x|\leq
1)}\\\nonumber
\lesssim&\int_{I_1}|t-s|^{-\frac{3}2+\sigma_1}\big\|(1+|x|^{-\sigma_1})\big(|x|^{-b}|u|^2u\big)\big\|_{L_x^1}\;ds\\\nonumber
\lesssim&|t-T+\epsilon^{-\alpha}|^{-\frac12+\sigma_1}\\\label{equ:pgeq2-1-1}
\lesssim&\epsilon^{(\frac12-\sigma_1)\alpha},\quad
t>T,
\end{align}
where
\begin{equation}\label{equ:alpha1}
\sigma_1=\begin{cases} 0\quad \text{if}\quad a\geq0\\
\sigma\quad\text{if}\quad a<0,
\end{cases}
\end{equation}
with $\sigma$ being as in Proposition \ref{pro:equivsobolev} and we have used the estimate for $a\geq0$
\begin{align*}
 \big\||x|^{-b}|u|^2u\big\|_{L^1_x}\lesssim &\big\||x|^{-b}\big\|_{L_x^{\frac{3}{b}-}(|x|\leq1)}\|u\|_{L_x^{\frac{9}{3-b}+}}^3+
 \big\||x|^{-b}\big\|_{L_x^{\frac{3}{b}+}(|x|\geq1)}\|u\|_{L_x^{\frac{9}{3-b}-}}^3\lesssim \|u\|_{H^1}^3\lesssim 1,
\end{align*}
while for $a<0$
\begin{align*}
\big\||x|^{-\sigma}\big(|x|^{-b}|u|^2u\big)\big\|_{L^1_x}\lesssim&\big\||x|^{-(\sigma+b)}\big\|_{L_x^{\frac{3}{\sigma+b}-}(|x|\leq1)}
\|u\|_{L_x^{\frac{9}{3-\sigma-b}+}}^3+
 \big\||x|^{-b}\big\|_{L_x^{\frac{3}{\sigma+b}+}(|x|\geq1)}\|u\|_{L_x^{\frac{9}{3-\sigma-b}-}}^3\\
 \lesssim& \|u\|_{H^1}^3\lesssim 1,
\end{align*}
since $3<\tfrac{9}{3-b},\tfrac{9}{3-\sigma-b}<6$ by the fact that $0<b<1$ and $0\leq\sigma<\tfrac12$.

By the same argument as above,  we obtain
\begin{align}\nonumber
\|F_1(t)\|_{L_x^\infty(|x|\geq
1)}\lesssim&\int_{I_1}\big\|e^{i(t-s)\la}\big(|x|^{-b}|u|^2u\big)(s)\big\|_{L_x^\infty(|x|\geq
1)}\;ds\\\nonumber
\lesssim&\int_{I_1}\big\|(1+|x|^{-\sigma_1})^{-1}e^{i(t-s)\la}\big(|x|^{-b}|u|^2u\big)\big\|_{L_x^\infty}\;ds
\\\nonumber
\lesssim&\int_{I_1}|t-s|^{-\frac{3}2+\sigma_1}\big\|(1+|x|^{-\alpha_1})\big(|x|^{-b}|u|^2u\big)\big\|_{L_x^1}\;ds\\\nonumber
\lesssim&|t-T+\epsilon^{-\alpha}|^{-\frac12+\sigma_1}\\\label{equ:pgeq2}
\lesssim&\epsilon^{(\frac12-\sigma_1)\alpha},\quad
t>T.
\end{align}
Interpolating this with \eqref{equ:l4l6est}, we deduce that
$$\|F_1(t)\|_{L_x^6(|x|\geq
1)}\lesssim \|F_1(t)\|_{L_x^2}^\frac13\|F_1(t)\|_{L_x^\infty(|x|\geq
1)}^\frac23\lesssim \epsilon^{\frac23(\frac12-\sigma_1)\alpha}.$$
Combining this  with \eqref{equ:pgeq2-1-1}, one has
\begin{equation}\label{equ:f1linfl6e}
  \|F_1\|_{L_t^\infty([T,\infty),L_x^6)}\lesssim \epsilon^{\frac23(\frac12-\sigma_1)\alpha}.
\end{equation}
Interpolating this with \eqref{equ:l4l6est} again yields
\begin{align}\label{equ:f1est}
\|F_1\|_{L_t^4([T,\infty),L_x^6)}\lesssim\|F_1\|_{L_t^\infty([T,\infty),L_x^6)}^\frac12
\|F_1\|_{L_t^2([T,\infty),L_x^6)}^\frac12 \lesssim\epsilon^{\frac13(\frac12-\sigma_1)\alpha}.
\end{align}

{\bf Estimate of the term $F_2(t,x)$:}
First, by the assumption \eqref{equ:masslim-1}, we
may choose $T>T_0$
\begin{equation}\label{equ:assumest}
\int_{\R^3} \chi_R(x)|u(T,x)|^2\;dx\leq \epsilon^2,
\end{equation}
where $\chi_R(x)\in C_c^\infty(\R^3)$ and
\begin{equation*}
\chi_R(x)=\begin{cases}
1\quad \text{if}\quad |x|\leq R,\\
0\quad \text{if}\quad |x|\geq 2R.
\end{cases}
\end{equation*}
On the other hand, combining the identity $\pa_t|u|^2=-2\nabla\cdot{\rm Im}(\bar{u}\nabla u)$ and integration by parts, H\"older's inequality, we obtain
\begin{equation*}
\Big|\pa_t\int \chi_R(x)|u(t,x)|^2\;dx\Big|\lesssim\frac1R.
\end{equation*}
Hence, for any $t\in I_2$
\begin{equation}\label{equ:tI2smal}
\int_{\R^3} \chi_R(x)|u(T,x)|^2\;dx\lesssim \epsilon^2+\frac{\epsilon^\alpha}{R}\lesssim \epsilon^2
\end{equation}
by choosing $R\geq \epsilon^{-(2-\alpha)}.$

Using H\"older's inequality, \eqref{equ:tI2smal}, Hardy's inequality, Sobolev embedding, we obtain for any $t\in I_2$
\begin{align}\nonumber
\big\| |x|^{-b}|u|^2u\big\|_{L_x^\frac65}\leq & \big\|\chi_R |x|^{-b}|u|^2u\big\|_{L_x^\frac65}+ \big\|(1-\chi_R) |x|^{-b}|u|^2u\big\|_{L_x^\frac65}\\\nonumber
\lesssim&\big\|\chi_R u\big\|_{L_x^\frac{6}{3-2b}}\big\||x|^{-b}|u|^2\big\|_{L_x^\frac{3}{1+b}}+R^{-b}\|u\|_{L_x^\frac{18}{5}}^3\\\nonumber
\lesssim&\big\|\chi_R u\big\|_{L_x^2}^{1-b}\big\|\chi_R u\big\|_{L_x^6}^{b}\big\||\nabla|^b(|u|^2)\big\|_{L_x^\frac{3}{1+b}}+R^{-b}\|u\|_{H^1}^3\\\nonumber
\lesssim&\epsilon^{1-b}\|u\|_{\dot{H}^1}^b\big\|\nabla (|u|^2)\big\|_{L_x^\frac32}+\epsilon^{(2-\alpha)b}\|u\|_{H^1}^3\\\nonumber
\lesssim&\epsilon^{1-b}\|u\|_{\dot{H}^1}^{2+b}+\epsilon^{(2-\alpha)b}\|u\|_{H^1}^3\\
\lesssim&\epsilon^{1-b}+\epsilon^{(2-\alpha)b}.
\end{align}
Hence, by H\"older's inequality in time, one has
\begin{equation}\label{equ:chiRes}
  \big\| |x|^{-b}|u|^2u\big\|_{L_t^2L_x^\frac65(I_2\times\R^3)}\lesssim \epsilon^{-\frac{\alpha}2}\big( \epsilon^{1-b}+\epsilon^{(2-\alpha)b} \big).
\end{equation}
On the other hand, by the same argument as in \cite[Lemma 2.7]{C21}, we have
$$\big\|\nabla \big(|x|^{-b}|u|^2u\big)\big\||_{L_x^\frac65}\lesssim \big\||u|^2\nabla u\big\|_{L_x^{\frac{6}{5-2b}\pm}}.$$
And so
\begin{align}\nonumber
  \big\|\nabla \big(|x|^{-b}|u|^2u\big)\big\||_{L_t^2L_x^\frac65(I_2\times\R^3)}\lesssim & \big\||u|^2\nabla u\big\|_{L_t^2L_x^{\frac{6}{5-2b}\pm}(I_2\times\R^3)}\\\nonumber
  \lesssim&\|u\|_{L_t^8L_x^{\frac{12}{3-2b}+}(I_2\times\R^3)}^2\|\nabla u\|_{L_t^4L_x^{3-}(I_2\times\R^3)}\\\label{equ:uepsilnab}
  \lesssim&|I_2|^\frac12\lesssim \epsilon^{-\frac{\alpha}{2}},
\end{align}
where we have used
$$\big\|\langle\nabla\rangle u\big\|_{L_t^q L_x^r(I\times\R^3)}\lesssim (1+|I|)^\frac1q,$$
which follows by using Strichartz estimate and continuous
argument.

Interpolating \eqref{equ:uepsilnab} with \eqref{equ:chiRes}, we obtain
\begin{equation}\label{equ:nab12es}
  \big\||\nabla|^\frac12 \big(|x|^{-b}|u|^2u\big)\big\||_{L_t^2L_x^\frac65(I_2\times\R^3)}\lesssim \epsilon^{-\frac{\alpha}{2}}\big( \epsilon^{1-b}+\epsilon^{(2-\alpha)b} \big)^\frac12.
\end{equation}
 Thus, we use Sobolev embedding, Strichartz
estimate, equivalence of Sobolev spaces (Lemma
\ref{pro:equivsobolev}) to get
\begin{align*}
  \|F_2\|_{L_t^4L_x^6([T,\infty)\times\R^3)}
  \lesssim&\big\||\nabla|^\frac12F_2\big\|_{L_t^4L_x^3([T,\infty)\times\R^3)}\\
   \lesssim&\big\|\la^\frac14F_2\big\|_{L_t^4L_x^3([T,\infty)\times\R^3)}\\
   \lesssim & \big\|\la^\frac14 \big(|x|^{-b}|u|^2u\big)\big\||_{L_t^2L_x^\frac65(I_2\times\R^3)}\\
  \lesssim & \big\||\nabla|^\frac12 \big(|x|^{-b}|u|^2u\big)\big\||_{L_t^2L_x^\frac65(I_2\times\R^3)}\\
  \lesssim& \epsilon^{-\frac{\alpha}{2}}\big( \epsilon^{1-b}+\epsilon^{(2-\alpha)b} \big)^\frac12\\
  \lesssim&\epsilon^{\frac{\alpha}2},
\end{align*}
by choosing
$$\alpha=\min\big\{\tfrac{1-b}2,\tfrac{2b}{2+b}\big\}.$$
Combining this with \eqref{equ:duhforeit}, \eqref{equ:linpsma} and \eqref{equ:f1est}, we get
\begin{equation}\label{equ:eitlinesmal}
  \|e^{i(t-T)\la}u(T)\|_{L_t^4L_x^6([T,\infty)\times\R^3)}\lesssim \epsilon+\epsilon^{\frac13(\frac12-\sigma_1)\alpha}+\epsilon^{\frac{\alpha}2}.
\end{equation}
And so \eqref{equ:linimpnonl} follows.

\end{proof}




\section{Virial-Morawetz estimates}\label{sec:vme}

 \subsection{Morawetz estimate}\label{subsec:mor}

First, we recall some standard virial-type identities.
 Given a weight $w:\R^3\to\R$ and a solution $u$ to $i\pa_tu+\Delta u= F$, we define
\[
V(t;w) := 2\Im \int_{\R^3} \bar u \nabla u \cdot \nabla w \,dx.
\]
It is easy to see that
\begin{align}\label{virial}
& \partial_{t}V(t;w) = \int_{\R^3} \Big[(-\Delta\Delta w)|u|^2 + 4\Re(\bar u_j u_k) w_{jk} +2\nabla w\cdot\{F,u\}_P \Big]\,dx,
\end{align}
 where $\{f,g\}_P=\Re(f\nabla\bar{g}-g\nabla\bar{f}).$ Especially, if $u$ solves \eqref{equ:nls}, then
\begin{align}\nonumber
  \partial_{t}V(t;w)=&\int_{\R^3} \Big[(-\Delta\Delta w)|u|^2 + 4\Re(\bar u_j u_k) w_{jk} +4a\frac{x\cdot \nabla w}{|x|^4}|u|^2 \Big]\,dx\\\label{equ:virialdef}
  &+\lambda b\int_{\R^3}\frac{x\cdot\nabla w}{|x|^{2+b}}|u|^4\;dx+\lambda\int_{\R^3}\Delta w(x)\frac{|u|^4}{|x|^b}\;dx.
\end{align}

Taking $w(x)=|x|^2$, one can obtain
the standard virial identity.

\begin{lemma}[Standard virial identity] Let $u$ be a solution to \eqref{equ:nls} with $\lambda=1$.  Then, there holds
\begin{equation}\label{equ:vtx2}
  \partial_{t} V(t;|x|^2) = 8\Bigl[\|u(t)\|_{\dot H_a^1}^2 +\lambda\int_{\R^3}\frac{|u|^4}{|x|^b}\;dx\Bigr].
\end{equation}
\end{lemma}

However, in our case, the virial action  $V(t;|x|^2)$ maybe infinite. Hence, we need a truncated version of the virial identity (cf. \cite{OgaTsu91}, for example). For $R>1$, we define $w_R(x)$ to be a smooth, non-negative radial function satisfying
\begin{equation}\label{equ:wrdef}
w_R(x)=\begin{cases} |x|^2 & |x|\leq \frac{R}2 \\ 2R|x| & |x|>R,\end{cases}
\end{equation}
with
\begin{equation}\label{equ:wrcond}
\pa_rw_R\geq0,~\pa_r^2w_R\geq0,~|\pa^\alpha w_R(x)|\lesssim_\alpha R|x|^{-|\alpha|+1},~|\alpha|\geq1.
\end{equation}
 This together with  \eqref{equ:virialdef} yields:
\begin{lemma}[Truncated virial identity]\label{lem:trancdef} Let $u$ be a solution to \eqref{equ:nls}  and let $R>1$. Then, we have
\begin{align*}
 \partial_{t} V(t;w_R)
  =& 8\int_{|x|\leq\frac{R}2}\Bigl[|\nabla u(t)|^2+a\tfrac{|u|^2}{|x|^2} +\lambda \tfrac{|u|^4}{|x|^b}\Bigr]\;dx  \\
& + \int_{|x|>R} \Bigl[8aR\tfrac{|u|^2}{|x|^3}+2(2+b)\lambda\tfrac{R}{|x|^{1+b}}|u|^{4}+\tfrac{8R}{|x|}(|\nabla u|^2-|\pa_ru|^2)\Bigr]\;dx \\
&  +\int_{\frac{R}2\leq|x|\leq R}\Bigl[4{\rm Re}\pa_{jk}w_R\bar{u}_j\pa_ku+O\bigl( \tfrac{R}{|x|^{1+b}}|u|^{4} + \tfrac{R}{|x|^3}|u|^2\bigr)\Bigr]\;dx.
\end{align*}
Moreover, it follows from \eqref{equ:wrcond} that $$\int_{\frac{R}2\leq|x|\leq R}\Bigl[4{\rm Re}\pa_{jk}w_R\bar{u}_j\pa_ku\Bigr]\;dx\geq 0.$$
\end{lemma}

\subsection{Proof of Proposition \ref{prop:virmorest}}

In this subsection, we aim to prove Proposition \ref{prop:virmorest} by the virial/Morawetz argument.
 By Lemma \ref{lem:trancdef}, we get
\begin{align}
\nonumber &\partial_{t} V(t;w_R)\\\nonumber
\geq& 8\int_{|x|\leq\frac{R}2}\Bigl[|\nabla u(t)|^2+a\tfrac{|u|^2}{|x|^2} +\lambda \tfrac{|u|^4}{|x|^b}\Bigr]\;dx  \\\nonumber
& + 8aR\int_{|x|>R}\tfrac{|u|^2}{|x|^3}\;dx  +\int_{|x|\geq\frac{R}2}\Bigl[O\bigl( \tfrac{R}{|x|^{1+b}}|u|^{4} + \tfrac{R}{|x|^3}|u|^2\bigr)\Bigr]\;dx\\\label{equ:mainter1}
\geq& 8\int_{|x|\leq\frac{R}2}\Bigl[|\nabla u(t)|^2+a\tfrac{|u|^2}{|x|^2} + \lambda\tfrac{|u|^4}{|x|^b}\Bigr]\;dx\\\nonumber& -\frac{C}{R^b}.
\end{align}
We define $\chi_R(x)$ to be a smooth function
\begin{equation}\label{equ:chirdef}
  \chi_R(x):=\begin{cases}
   1\quad |x|\leq R/4\\
   0\quad |x|\geq R/2.
  \end{cases}
\end{equation} Note that
\begin{equation}\label{equ:identity}
\int\chi_R^2|\nabla u|^2\;dx=\int\Big[|\nabla(\chi_Ru)|^2+\chi_R\Delta(\chi_R)|u|^2\Big]\;dx,
\end{equation}
we get
\begin{align}\nonumber
\eqref{equ:mainter1}\geq&8\int\chi_R^2\Bigl[|\nabla u(t)|^2+a\tfrac{|u|^2}{|x|^2}+ \lambda\tfrac{|u|^4}{|x|^b}\Bigr]\;dx+\int O\big(\tfrac{|u|^2}{R^2}\big)\;dx+\int_{\frac{R}4\leq|x|\leq\frac{R}2}O\big(\tfrac{|u|^4}{R^b}\big)\;dx \\\nonumber
\geq&8\Big[\big\|\chi_Ru\big\|_{\dot{H}^1_a}^2+\lambda \int_{\R^3}\tfrac{|\chi_Ru|^4}{|x|^b}\;dx\Big]\\\nonumber
&+\int\chi_R\Delta(\chi_R)|u|^2\;dx+\int O\big(\tfrac{|u|^2}{R^2}\big)\;dx-\frac{C}{R^b}\\\nonumber
\geq&c \int_{\R^3}\tfrac{|\chi_Ru|^4}{|x|^b}\;dx-\frac{C}{R^b}-\frac{C}{R^2}\\\label{equ:mainter2}
\geq&c \int_{|x|\leq\frac{R}4}\tfrac{|u|^4}{|x|^b}\;dx-\frac{C}{R^b}-\frac{C}{R^2},
\end{align}
where we have used the fact that
$$\big\|\chi_Ru\big\|_{\dot{H}^1_a}^2+ \lambda\int_{\R^3}\tfrac{|\chi_Ru|^4}{|x|^b}\;dx\geq c\int_{\R^3}\tfrac{|\chi_Ru|^4}{|x|^b}\;dx.$$
It is easy to see that it holds when $\lambda=1$. While for $\lambda=-1$, this follows by the same argument as in Proposition \ref{P:coercive} $(ii)$.

Thus, plugging \eqref{equ:mainter2} into \eqref{equ:mainter1}, we obtain
\begin{align*}
\partial_{t} V(t;w_R)\geq c\int_{|x|\leq\frac{R}4}\tfrac{|u|^4}{|x|^b}\;dx-\frac{C}{R^b}.
\end{align*}
Using the fundamental theorem of calculus on an interval $[0,T]$,  we have
$$\int_0^T\int_{|x|\leq\frac{R}4}\tfrac{|u|^4}{|x|^b}\;dx\;dt\lesssim\sup_{t\in[0,T]}|V(t,w_R)|+\frac{T}{R^b}.$$
From H\"older's inequality, \eqref{virial} and potential energy bound, we get
$$\sup_{t\in\R}|V(t,w_R)|\lesssim R.$$
 Thus, we obtain
$$\frac1{T}\int_0^T\int_{|x|\leq \frac{R}4}|u(t,x)|^{p+1}\;dx\;dt\lesssim\frac{R}{T}+\frac1{R^b},$$
which is accepted. Hence we conclude the proof of Proposition \ref{prop:virmorest}.




\section*{Appendix}

In this appendix, we will establish the interaction Morawetz estimate for the solution to \eqref{equ:nls} with $\lambda=1$. As an application, we give a simple proof for the scattering theory for \eqref{equ:nls} with $\lambda=1$ but $a>0$.

\begin{theorem}[Interaction Morawetz estimate]\label{thm:ime}
Let  $a>0,\;0<b<1$ and $u:\;\R\times\R^3\to\C$ solve \eqref{equ:nls} with $\lambda=1$ and $u_0\in H^1(\R^3)$. Then, there holds
\begin{equation}\label{equ:ime}
  \|u\|_{L_{t,x}^4(\R\times\R^3)}\leq C\big(M(u_0),E(u_0)\big).
\end{equation}
Interpolating this with $\|u\|_{L_t^\infty L_x^6}\lesssim \|u\|_{L_t^\infty \dot{H}^1}\lesssim \sqrt{E(u_0)}$ implies
$$\|u\|_{L_t^6 L_x^\frac{9}{2}(\R\times\R^3)}\leq  C\big(M(u_0),E(u_0)\big).$$
As a consequence, the global solution $u$ to \eqref{equ:nls} with $\lambda=1$ scatters.
\end{theorem}

\begin{remark}
For the higher dimension case $d\geq4$, one can also obtain the interaction Morawetz estimate
\begin{equation}\label{equ:intmeshigh}
  \big\||\nabla|^{-\frac{d-3}{4}}u\big\|_{L_{t,x}^4(\R\times\R^d)}\leq C\big(M(u_0),E(u_0)\big),
\end{equation}
where $u:\R\times\R^d\to\C$ solves $i\pa_tu-\la u=|x|^{-b}|u|^{p-1}u$ with $a>-\tfrac{(d-2)^2}4+\tfrac14$ and $b>0$. Compared with dimension three case and higher dimension case, we can get the interaction Morawetz estimate for some negative $a$ in higher dimension cases.

\end{remark}

To prove Theorem \ref{thm:ime}, we first
consider that the function $u(t,x)$ solves
\begin{equation}\label{equ:nlsF}
i\pa_tu+\Delta u=F(t,x),\quad (t,x)\in\R\times\R^3.
\end{equation}
Define Morawetz action
\begin{equation}\label{equ:moract}
M_w(t):=2{\rm Im}\int_{\R^3}\nabla w(x)\cdot \nabla
u(x)\bar{u}(x)\;dx.
\end{equation}
A simple computation shows
\begin{lemma}[Morawetz identity]\label{lem:moriden}
There holds
\begin{align*}
\frac{d}{dt}M_w(t)=\int_{\R^3}(-\Delta\Delta
w)|u|^2dx+4\int_{\R^3}w_{jk}\Re(\pa_j\bar{u}\pa_ku)dx+2\int_{\R^3}w_j(x)\{F,u\}_P^jdx
\end{align*}
where $\{f,g\}_P=\Re(f\nabla\bar{g}-g\nabla\bar{f}),$ and repeated
indices are implicitly summed.  Especially, if $u$ solves \eqref{equ:nls} with $\lambda=1$, then
\begin{align}\nonumber
  \frac{d}{dt}M_w(t)=&\int_{\R^3} \Big[(-\Delta\Delta w)|u|^2 + 4\Re(\bar u_j u_k) w_{jk} +4a\frac{x\cdot \nabla w}{|x|^4}|u|^2 \Big]\,dx\\\label{equ:virialdefint}
  &\qquad+ b\int_{\R^3}\frac{x\cdot\nabla w}{|x|^{2+b}}|u|^4\;dx+\int_{\R^3}\Delta w(x)\frac{|u|^4}{|x|^b}\;dx.
\end{align}
\end{lemma}

Taking $w(x)=|x|$, one can obtain
the standard Morawetz inequality.

\begin{lemma}[Classical Morawetz inequality]\label{lem:clamorz}
Let  $a>0,\;0<b<1$ and $u:\;\R\times\R^3\to\C$ solve \eqref{equ:nls} with $\lambda=1$ and $u_0\in H^1(\R^3)$. Then, there holds
\begin{equation}\label{equ:clamorze}
  \int_{\R}\int_{\R^3}\Big(\frac{|u(t,x)|^2}{|x|^3}+\frac{|u(t,x)|^4}{|x|^{1+b}}\Big)\;dx\;dt\leq  C\big(M(u_0),E(u_0)\big).
\end{equation}
\end{lemma}

Define
the  Morawetz action center $y$ as
\begin{equation}\label{equ:moractre1}
M_w^y(t):=2{\rm Im}\int_{\R^3}\nabla w(x-y)\cdot \nabla
u(x)\bar{u}(x)\;dx.
\end{equation}

Now, we define the interaction Morawetz action by
\begin{align}\label{equ:moractre12}
M^{Int}(t):=&\int_{\R^3}|u(y)|^2 M_w^y(t)\;dy\\\nonumber
 =&2{\rm
Im}\iint_{\R^3\times\R^3}|u(y)|^2\nabla w(x-y)\cdot \nabla
u(x)\bar{u}(x)\;dx\;dy.
\end{align}
Note that
$$\pa_t(|u|^2)=2{\rm Im}(-\Delta u\bar{u}+F\bar{u})=-2\pa_j{\rm Im}(\pa_ju\bar{u})+2{\rm Im}(F\bar{u}),$$
this together with Lemma \ref{lem:moriden} implies
\begin{align}\nonumber
\frac{d}{dt}M^{Int}(t)=&\int_{\R^3}\pa_t(|u(y)|^2)
M_w^y(t)\;dy+\int_{\R^3}|u(y)|^2\frac{d}{dt}M_w^y(t)\;dy\\\label{equ:main1}
=&-4\iint {\rm Im}(\pa_ju\bar{u})(y) w_{jk}(x-y){\rm
Im}(\pa_ku\bar{u})(x)\;dx\;dy\\\label{equ:main2} &+4\iint {\rm
Im}(F\bar{u})(y)\nabla w(x-y)\cdot{\rm Im}(\nabla
u\bar{u})(x)\;dx\;dy\\\label{equ:main3} &+\iint (-\Delta\Delta
w)(x-y)|u(x)|^2|u(y)|^2\;dx\;dy\\\label{equ:main4} &+4\iint
|u(y)|^2w_{jk}(x-y)\Re(\pa_j\bar{u}\pa_ku)(x)\;dx\;dy\\\label{equ:main4}
&+2\iint |u(y)|^2w_j(x-y)\{F,u\}_P^j(x)\;dx\;dy.
\end{align}
Hence, taking $w(x-y)=|x-y|$, one can obtain the interaction Morawetz identity.
\begin{lemma}[Interaction Morawetz identity]\label{lem:intmoridet}
Assume that $u:\;\R\times\R^3\to\C$ solves \eqref{equ:nls} with $\lambda=1$, then
\begin{align}\nonumber
&\frac{d}{dt}M^{Int}(t)\\\label{equ:main1-1}
=&-4\iint {\rm Im}(\pa_ju\bar{u})(y)\Big(\frac{\delta_{jk}}{|x-y|}-\frac{(x-y)_j(x-y)_k}{|x-y|^3}\Big){\rm
Im}(\pa_ku\bar{u})(x)\;dx\;dy\\\label{equ:main2-1} &+4\iint
|u(y)|^2\Big(\frac{\delta_{jk}}{|x-y|}-\frac{(x-y)_j(x-y)_k}{|x-y|^3}\Big)\Re(\pa_j\bar{u}\pa_ku)(x)\;dx\;dy
\\\label{equ:main3-1} &+c\int_{\R^3} |u(x)|^4\;dx\\\label{equ:main4-2}
&+4\iint |u(y)|^2\frac{x}{|x|^4}\cdot\frac{x-y}{|x-y|}|u(x)|^2\;dx\;dy\\\label{equ:main5-2}
&+b\iint |u(y)|^2\frac{x}{|x|^{2+b}}\cdot\frac{x-y}{|x-y|}|u(x)|^4\;dx\;dy\\\label{equ:main6-2}
&+2\iint |u(y)|^2\frac{|u(x)|^4}{|x|^{b}}\;dx\;dy.
\end{align}
\end{lemma}

By the same argument as in Killip-Visan \cite{KV}, we know that
$$\eqref{equ:main1-1}+\eqref{equ:main2-1}\geq0.$$
On the other hand, by the definition of $M^{Int}(t)$, we know that
$$|M^{Int}(t)|\lesssim \|u(t)\|_{L_x^2}^3\|u(t)\|_{\dot{H}^1}\leq  C\big(M(u_0),E(u_0)\big).$$
Hence, we obtain
\begin{align*}
  &\int_{\R}\int_{\R^3}|u(t,x)|^4\;dx\;dt\\
  \leq & 2\sup_{t\in\R}|M^{Int}(t)|+\|u_0\|_{L_x^2}^2\int_{\R}\int_{\R^3}\Big(\frac{|u(t,x)|^2}{|x|^3}+\frac{|u(t,x)|^4}{|x|^{1+b}}\Big)\;dx\;dt\\
  \lesssim&C\big(M(u_0),E(u_0)\big),
\end{align*}
where we have used Lemma \ref{lem:clamorz}. Therefore, we conclude the proof of Theorem \ref{thm:ime}.

\begin{center}

\end{center}

\end{document}